\documentclass{amsart}
\usepackage{amsthm}
\usepackage{amssymb}
\usepackage{amscd}
\usepackage{a4wide}
\usepackage[latin1]{inputenc}
\newtheorem*{thm}{Theorem}
\newtheorem*{prop}{Proposition}
\newtheorem*{cor}{Corollary}
\newtheorem*{lem}{Lemma}
\newtheorem*{defn}{Definition}

\newcommand{\cB}{\mathcal B}

\newcommand{\AH}{{A\# H}}
\newcommand{\sAH}{\sigma[{_{\AH}A}]}
\newcommand{\sAB}{\sigma[{_BA}]}
\newcommand{\AAH}{{A^e\bowtie H}}
\newcommand{\sAAH}{\sigma[{_{\AAH}A}]}
\newcommand{\EndX}[2]{{\mathrm{End}}_{#1}\:(#2)}
\newcommand{\HomX}[3]{{\mathrm{Hom}}_{#1}\:(#2, #3)}
\newcommand{\Ker}[1]{{\mathrm{Ker}}\:(#1)}

\newcommand{\Rad}[1]{{\mathrm{Rad}}\:(#1)}
\newcommand{\ov}[1]{{\overline{#1}}}

\newcommand{\EA}{\widehat{A}}
\title{Regular and Biregular module algebras}
\author{Christian Lomp}
\address{University of Porto, Department of pure Mathematics, Rua Campo Alegre 687, 4169-007 Porto (PORTUGAL)\\
\textit{clomp@fc.up.pt}}

\thanks{This work was partially supported by Centro de Matem\'{a}tica da Universidade do Porto (CMUP), financed by FCT
(Portugal) through the programs POCTI (Programa Operacional Ci\^{e}ncia, Tecnologia, Inova\c{c}\~{a}o) and POSI
(Programa Operacional Sociedade da Informa\c{c}\~{a}o), with national and European community structural funds.}

\begin{document}
\maketitle
\begin{abstract}
Motivated by the study of von Neumann regular skew groups as carried out by Alfaro, Ara and del Rio in \cite{AlfaroAraRio} we investigate regular and biregular Hopf module algebras. If  $A$ is an algebra  with an action by an affine Hopf algebra $H$, then any $H$-stable left ideal of $A$ is a direct summand if and only if $A^H$ is regular and the invariance functor $(-)^H$ induces an equivalence of $A^H$-Mod to the Wisbauer category of $A$ as $\AH$-module. Analogously we show a similar statement for the biregularity of $A$ relative to $H$ where $A^H$ is replaced by $R=Z(A)\cap A^H$ using the module theory of $A$ as a module over $\AAH$ the envelopping Hopf algebroid of $A$ and $H$. We show that every two-sided $H$-stable ideal of $A$ is generated by a central $H$-invariant idempotent if and only if $R$ is regular and $A_m$ is $H$-simple for all maximal ideals $m$ of $R$. Further sufficient conditions are given for $\AH$ and $A^H$ to be regular.
\end{abstract}

\section{Introduction}

Motivated by the study of  von Neumann regular of skew group rings by Alfaro et al.  in \cite{AlfaroAraRio} and by the studies of the regularity of fix rings by  Goursad et all in \cite{GOPV} we look at the regularity of Hopf module algebras, their smash products and their subrings of invariants. To achieve our goal we will work in the following more general setting:

Let $k$ be a commutative ring. An extension $A\subseteq B$ of $k$-algebras is said to have an additional module structure if there exists a ring homomorphism $\Psi:B\rightarrow \EndX{k}{A}$ such that $\Psi(a)=L_a$ for all $a\in A$, where $L_a$ denotes the left multiplication of $a$ on $A$. Then $A$ is a cyclic left $B$-module with $B$-action $b\cdot a := \Psi(b)(a)$ for all $b\in B, a\in A$. Moreover $\alpha:B  \longrightarrow A$ with $(b)\alpha = b\cdot 1$ is an epimorphism of left $B$-modules. Note that we will write homomorphisms oposite of scalars. Furthermore $\phi:\EndX{B}{A}\longrightarrow A$ with $\phi(f)=(1)f$ defines a ring homomorphism whose image is denote by $A^B$. In particular
$$A^B = \{a \in A \mid \forall b\in B: b\cdot a = (b)\alpha a\} = \{a \in A \mid \forall b\in B \:\: \forall a'\in A: b\cdot (a'a) = (b\cdot a')a\}.$$
Defining for any $B$-module $M$:
$$M^B = \{m\in M\mid \forall b\in B \:\: \forall a \in A: b\cdot (am) = (b\cdot a)m\}$$ one also has functorial isomorphisms
$$ \HomX{B}{A}{M} \longrightarrow M^B \:\:\: f \mapsto (1)f$$
such that $\HomX{B}{A}{ - }$ and $( - )^B$ are isomorphic functors (see \cite{Lomp} for details).
In the terminology of \cite{Caenepeel2}, $B$ is an $A$-ring with a right grouplike character.

Examples of the described situation are abundant in the theory of Hopf algebra actions where a Hopf algebra $H$ (or more general a weak Hopf algebra) acts on an algebra $A$ and $A\subseteq B=\AH$ is an extension with additional module structure. This also includes group action and Lie actions. Further examples are given by the envelopping algebra $A\subseteq A^e$  or more generally by the envelopping Hopf algebroid $A^e\bowtie H$ as defined in \cite{Lomp} (see also \cite{Lomp05} or \cite{ConnesMoscovici}), $k$-algebras $A$ with an involution $*$ with $B=A^e * G$ where $G=\langle \sigma \rangle$ is group generated by the automorphism $\sigma$ of $A^e$ defined by $\sigma(a\otimes b)=b^* \otimes a^*$  or certain extensions $A\subseteq B$ arrising in the study of Banach algebras (see Cabrera et al. \cite{Cabrera}).

In this paper we will characterize regular and biregular $H$-module algebras, generalising some known results on the regularity skew group rings.

All rings will be associative and unital. Ring homomorphisms are supposed to respect the unit.
Throughout the text $k$ will denote a commutative ring and $A$ a $k$-algebra. We denote by $A^e := A \otimes A^{op}$ the enveloping algebra of $A$  whose multiplication is defined as $(a\otimes b)(a'\otimes b')=aa'\otimes b'b$.

\section{Regular Modules}
John von Neumann defined a ring $R$ to be regular  if for any element $a\in R$ there exists an element $b\in R$ such that $a=aba$. He showed in \cite{Neumann} that $R$ is regular if and only if every  cyclic (finitely generated) left (right) ideal of $R$ is a direct summand. Later Auslander proved that the regularity of a ring can also be characterised by the property that any module is fla or equivalently that any submodule of a module is pure. Several author's have transfered the regularity condition to modules. A.Tuganbaev in \cite{Tuganbaev} calls a left $R$-module $M$ regular if any cyclic (finitely generated) submodule is a direct summand using the lattice theoretical approach, while J.Zelmanowitz in \cite{Zelmanowitz72} followed the original elementwise definition of von Neumann and called a left $R$-module $M$ regular if for any  $m\in M$ there exists $f\in \HomX{R}{M}{R}$ such that $(m)fm = m$.

The module theoretic version of Auslander's charcterisation had been carried out by Fieldhouse \cite{Fieldhouse} where he called a left $R$-module regular if any of its submodule is pure in the sense of P.M.Cohen. R.Wisbauer \cite{Wisbauer_regular} used his ideas to define regularity for nonassociative rings (see also \cite[Chapter 34]{wisbauer}): Let $R$ be an arbitrary ring and $M$ a left $R$-module.  The Wisbauer category $\sigma[M]$ is the subcategory of $R$-Mod whose objects are the submodules of $M$-generated modules, i.e. submodules of factor modules of direct sums of copies of $M$. A module $P \in \sigma[M]$ is called finitely presented in $\sigma[M]$ if $P$ is finitely generated and every exact sequence in $\sigma[M]$:
$$\begin{CD}
0 @>>> K @>>> L @>>> P @>>> 0  
\end{CD}$$
with $L$ finitely generated implies $K$ to be finitely generated.
Note that $P$ might be finitely presented in $\sigma[M]$ but not in $R$-Mod, for example take any simple module $P=M$.
A short exact sequence in $\sigma[M]$ is called {\it pure} if any finitely presented module in $\sigma[M]$ is projective with respect to this sequence and a module $N \in \sigma[M]$ is called {\it flat} in $\sigma[M]$ if any short exact sequence 
$$\begin{CD}
0 @>>> K @>>> L @>>> N @>>> 0  
\end{CD}$$
in $\sigma[M]$ is pure. Finally $M$ is called {\it regular} if any module in $\sigma[M]$ is flat or equivalently if any short exact sequence in $\sigma[M]$ is pure.

\subsection{Relative regularity}
Let $A\subseteq B$ be an extension with additional module structure. Our first aim will be to characterise $A$ as a regular $B$-module.

\begin{prop}\label{RegularExtensionsModuleStructure} Let $A\subseteq B$ be an extension with additional module structure.
The following statments are equivalent:
\begin{enumerate}
  \item[(a)] $A$ is regular and finitely presented in $\sigma[{_BA}]$;
  \item[(b)] Every left $B$-stable ideal that is finitely generated as left $B$-module is a direct summand and $A$ is finitely presented in
$\sigma[{_BA}]$.
  \item[(c)] $A^B$ is von Neumann regular and $A$ is a generator in
$\sigma[{_BA}]$.
  \item[(d)] $A^B$ is von Neumann regular and $()^B=\HomX{B}{A}{-}: \sigma[{_BA}]
\rightarrow A^B-Mod$ is an equivalence of categories.
\end{enumerate}
In this case $A$ is a projective generator in $\sigma[{_BA}]$.
\end{prop}

\begin{proof}
 $(a)\Leftrightarrow (b)$ follows from \cite[37.4]{wisbauer}

$(a)\Rightarrow (c)$ Since $A$ is finitely presented and regular in $\sigma[{_BA}]$, it is projective in $\sigma[{_BA}]$.  By 
As $A$ is a cyclic $B$-module,  by \cite[37.8]{wisbauer}, $A$ is a (finitely generated) projective generator in $\sigma[{_BA}]$ and $A^B$ is regular.

$(c)\Leftrightarrow (d)$ is clear.

$(c) \Rightarrow (b)$ Since $_BA$ is cyclic and a generator in $\sigma[{_BA}]$ and since $A^B\simeq \EndX{B}{A}$ is regular and thus $A$ is a faithfully flat $A^B$-module, we have by \cite[18.5(2)]{wisbauer} that $A$ is self-projective (and hence projective in $\sigma[_BA]$). This implies that $A$ is also finitely presented in $\sigma[{_BA}]$. If $U$ is a finitely generated $B$-stable left ideal of $A$, then $U=AI$ for $I=U^B$. $U$ being finitely generated as $B$-module, implies $I$ being finitely generated as right  ideal of $A^B$. Thus $I=A^Be$ for some idempotent $e$ and $U=AI=Ae$ is a direct summand of $A$, i.e. $A$ is regular by \cite[37.4]{wisbauer}.

\end{proof}


\subsection{Relative biregularity }
If $A\subseteq B$ is an extension  with additional modules structure $\Psi:B\rightarrow \EndX{k}{A}$,  we might identify $B$ with its image in $\EndX{k}{A}$ seeing it as an extension of the subalgebra generated by left multiplications of $A$. In order to study the two-sided $B$-stable ideals we might enlarge $B$ by considering $B'=\langle B \cup M(A) \rangle\subseteq \EndX{k}{A}$.
Note that all $B$-submodules of $A$ are two-sided  and  $A^B \subseteq Z(A)$ if $M(A)\subseteq B \subseteq \EndX{k}{A}$.
A $B$-stable ideal $I$ is called prime if  $JK \subseteq I$ implies $J\subseteq I$ or $K\subseteq I$ for any $B$-stable ideals $J$ and $K$.
$I$ is semiprime if it is the intersection of prime $B$-stable ideals. $A$ is $B$-semiprime if $0$ is a prime as $B$-stable ideal or equivalently $A$ does not contain any non-zero nilpotent $B$-stable ideal (see \cite[2.3]{Lomp05}). If a cyclic  $B$-stable ideal  $B\cdot a$ is a direct summand of $A$, then there exists an idempotent $e \in A^B$ with $B\cdot a = B\cdot e = Ae$ and $A = Ae \oplus A(1-e)$.
A $k$-algebra $A$ is called {\it $B$-biregular} if every cyclic $B$-stable ideal is a direct summand of $A$.
In particular Proposition \ref{RegularExtensionsModuleStructure} applies to get a characterisation of $B$-biregular algebras $A$ in case $A$ is finitely presented in $\sigma[_BA]$, namely that $A$ is $B$-biregular if and only if $A^B$ is a von Neumann regular ring with $(-)^B:\sigma[_BA]\rightarrow A^B$-Mod being an equivalence.

\subsection{Properties of relative biregular algebras}
In the next two subsections, we intend to characterise $B$-biregular algebras $A$ without assuming that $A$ is finitely presented in $\sAB$.

\begin{prop}\label{PropertiesBiregular}
Let $M(A)\subseteq B \subseteq \EndX{k}{A}$. Suppose that $A$ is $B$-biregular. Then
\begin{enumerate}
\item $A^B$ is von Neumann regular and $A$ is $B$-semiprime.
\item $A$ is a $A^B$-Ideal Algebra, i.e. the map $ I \mapsto IA $ is a bijection between the ideals $I$ of $A^B$ and the $B$-stable ideals of $A$, whose  inverse is given by $N \mapsto \mathrm{Ann}_{A^B}(A/N) \simeq \HomX{B}{A/N}{A}$.
\item Every finitely generated $B$-stable ideal of $A$ is cyclic and is generated by some central idempotent in $A^B$.
\item For any $B$-stable ideal $I$ of $A$, also $A/I$ is $B/I$-biregular.
\item Every $B$-stable ideal of $A$ is idempotent and equals the intersection of maximal $B$-stable ideals.
\item Every prime $B$-stable ideal is maximal.
\end{enumerate}
\end{prop}
 
\begin{proof}
(1) Let  $f\in \EndX{B}{A}$, then  $(A)f=B(1)f$ is a direct sumand in $A$ by hypothesis, i.e. $(A)f=Ae$ with $e^2=e\in A^B\subseteq Z(A)$. Since  $A(1-e)\subseteq \Ker{f}\subseteq l.ann((A)f)=A(1-e)$ also the kernel of $f$ is a  direct  summand,. Hence by  \cite[7.6]{wisbauer96}, $\EndX{B}{A}$ and thus $A^B$ is regular. Since no cyclic $B$-stable ideal is nilpotent, $A$ is $B$-semiprime.

(2) $A$ generates all cyclic $B$-stable ideals, i.e. $_BA$ is a self-generator and since  $A^B$ is regular by (1),   $_BA$ is intern-projective by  \cite[5.6]{wisbauer96}. Since $A$ is a cyclic $B$-module, the claim then follows by \cite[5.9]{wisbauer96}.

(3) Let $Ae$ and $Af$ be cyclic  $B$-stable ideals with idempotents $e,f \in A^B$. Then  $Ae + Af = A(e+f-ef) = A(e\uplus f)$, where $\uplus$ is the addition in the boolean ring of idempotents $B(A^B)$.

(4) By (2), every $B$-stable ideal $I$ can be written as $I=JA$ with $J$ ideal in $Z:=A^B$.
Hence the canonical projection $A=A\otimes_Z Z \rightarrow A/I \simeq A\otimes Z/J$ can be understood as the tensoring of the canonical projection of $Z\rightarrow Z/J$ by  $A \otimes_Z -$, which respects direct sums.

(5) For every cyclic $B$-stable ideal  $B\cdot x = Ae$ we have $(Ae)^2=A^2e^2=Ae$ . Hence $B\cdot x$ and thus any $B$-stable ideal is idempotent. Since there are no small $B$-submodules in $A$, we have $\Rad{_BA}=0$ and $0$ is the intersection of maximal $B$-stable ideals. By (4) we can use this argument to each $A/I$. 

(6) Suppose $A$ is $B$-prime and $B$-biregular. Let $0\neq I=Ae$ be a cyclic $B$-stable ideal with idempotent $e$.
As $A(1-e)$ is a $B$-stable ideal  with  $A(1-e)I = 0$, we have  $A(1-e)=0$, i.e. $I=A$ and $A$ is $B$-simple.
\end{proof}

\subsection{Characterisation of relative biregularity}
The next Proposition characterises biregular extensions $A \subseteq M(A)\subseteq B \subseteq \EndX{k}{A}$. Denote by $\mathrm{Max}(A^B)$ the spectrum of maximal ideals of $A^B$ and by $A_m$ the localisation of $A$ by a maximal ideal $m$ of $A^B$. Note that if $A^B$ is regular, then $A_m=A/mA$ by \cite[17.7]{wisbauer96} and in particular since $mA$ is $B$-stable, we might consider $B\subseteq \EndX{k}{A/mA} = \EndX{k}{A_m}$. We say that $A$ is $B$-simple if $0$ and $A$ are the only $B$-stable ideals of $A$.

\begin{thm}\label{CharacterisationBiregularity}
The following statements are equivalent for an extension $M(A)\subseteq B \subseteq \EndX{k}{A}$.
\begin{enumerate}
\item[(a)] $A$ is $B$-biregular;
\item[(b)] $A^B$ is regular and every maximal $B$-stable ideal $M$ of $A$ is of the form $M=AM^B$.
\item[(c)] $A^B$ is regular and $A_m$ is $B$-simple for all $m\in \mathrm{Max}(A^B)$.
\end{enumerate}
\end{thm}

\begin{proof}
$(a) \Rightarrow (b)$ the properties $(i-iii)$ follow from Proposition \ref{PropertiesBiregular} and $(iv)$ follows from the fact if $A$ is $B$-biregular then for any $x\in A: l.ann_A(Bx)=A(1-e)$ with $e^2=e \in A^B$ is already a $B$-ideal.

$(b)\Rightarrow (c)$:
Let $m$ be a maximal ideal of $A^B$ and let $M$ be a maximal $B$-stable ideal containing $mA\subseteq M$. Since $M=M^BA$ we have 
$$m\subseteq (mA)^B \subseteq M^B$$ which implies $M^B=m$ since $M\neq A$. Thus $mA=M$ and $A/M=A/mA=A_m$ is $B$-simple.

$(c)\Rightarrow (a)$ Let $I$ be any $B$-stable ideal $I$ of $A$. Then $I^BA \subseteq I$ and $$(I^BA)_m = (I\cap A^B)_m A_m = (I_m \cap A^B_m)A_m.$$
If $I_m=A_m$ then $I^B_m = I_m\cap A^B_m = A^B_m$ and hence $(I^BA)_m=I_m$. If $I_m \neq A_m$, then $I_m=0_m$ and therefore $I^B_m = 0_m$, i.e. $(IA^B)_m=I_m$. Since this holds for any maximal ideal $m$ of $A^B$, we get $I=I^BA$ which shows that $A$ is a self-generator as $B$-module.

\end{proof}

\subsection{Regular subring of invariants}
Assume again that $A\subseteq B$ is any extension with additional module structure. 
In order to determine when the subring of invariants $A^B$ is regular, we need first to borrow another notion from module theory.

\begin{defn} A left $R$-module $M$ is called semi-projective, if every
diagram:
$$\begin{CD}
@.M\\
@.@VgVV\\
M@>f>> N @>>> 0
\end{CD}$$
with $N\subseteq M$ can be completed by an endomorphism $h \in
S:=\EndX{R}{M}$ such that $hf=g$.
As it is easily seen:
$M$ is semi-projective if and only if $\HomX{R}{M}{Mf}=Sf$ for all $f\in S$.
\end{defn}

Hence $A$ is  semi-projective as left $B$-module if $\forall x\in A^B: (Ax)^B = A^Bx$.

\begin{prop}\label{RegularRelativeFixRing} Let $A\subseteq B$ be an extension with additional module structure. Then $A^B$ is  von Neumann regular if and only if
\begin{enumerate}
 \item $A$ is semi-projective as left $B$-module and
\item every cyclic left ideal generated by an $B$-invariant element $x\in A^B$ is a direct summand of $A$ as $B$-module.
\end{enumerate}

\end{prop}

\begin{proof}
If $A^B\simeq \EndX{B}{A}$ is regular, then $_BA$ is  semi-projective by  \cite[5.9]{wisbauer96}. Furthermore since the images of $B$-linear maps are direct summands and are precisely the cyclic $B$-stable left ideals generated by a $B$-invariant element we are done.

On the other hand assume that $_BA$ is semi-projective.
Let $0\neq x \in A^B$ then $B\cdot x = Ax$ is a direct summand of $A$ as left $B$-module by hypothesis.
Thus $A=Ax \oplus I$ as left $B$-modules. But then
$$A^B = (1)(\HomX{B}{A}{Ax}\oplus \HomX{B}{A}{I}) = (Ax)^B \oplus I^B = A^Bx \oplus I^B.$$
Hence every cyclic left ideal of $A^B$ is a direct summand, i.e. $A^B$ is von Neumann regular.
\end{proof}

\subsection{Large subring of invariants}

If  $A$ is finitely presented and regular in $\sAB$, then by Propositon \ref{RegularExtensionsModuleStructure} it is a projective generator.in $\sAB$. Weakening the generator conditon J.Zelmanowitz called a left $R$-module $M$ {\it retractable} if $\HomX{R}{M}{N}\neq 0$ for all non-zero submodules $N\subseteq M$. 
For a module algebra extension $A\subseteq B$ we say that $A^B$ is {\it large} in $A$ if $I\cap A^B \neq 0$ for all $B$-stable left ideals of $A$ or eqivalently if $A$ is a retractable $B$-module. A classical theorem of Bergmann and Isaacs says that if  finite group $G$ acts on an algebra $A$ such that $A$ is $G$-semiprime and has no $|G|$-torsion, then $R^G$ is large in $R$.

A purely module theoretical result by J.Zelmanowitz from \cite{Zelmanowitz} says now in our language:

\begin{lem}\label{Zelmanowitz_polyform} Let $A$ be projective in $\sAB$ and $A^B$ large in $A$, then
\begin{enumerate}
\item If $A^B$ is left self-injective, then $A$ is a self-injective left $B$-module.
\item If $A^B$ is von Neumann regular, then $A$ is a non-singular in $\sAB$, i.e. if $K \subseteq L$ is an essential extension in $\sAB$, then $\HomX{B}{L/K}{A}=0$.
\end{enumerate}
\end{lem}

\begin{proof}
Zelmanowitz calls  a left $R$-module $M$ {\it fully retractable} if $\HomX{R}{M}{N}g \neq 0$ for any $0\neq g\in \HomX{R}{N}{M}$ and submodule $N \subseteq M$. It is easy to see that self-projective retractable modules are fully retractable. Zelmanowitz proves  in \cite[Proposition on page 567]{Zelmanowitz} that $M$ is self-injective if $M$ is fully retractable and left $\EndX{R}{M}$ self-injective. Property (2) follows from \cite[Corollary on page 568]{Zelmanowitz}.
\end{proof}

Note that a module $M$ is non-singular in $\sigma[M]$ if and only if it is ``polyform'' in the sense of J.Zelmanowitz (see \cite{wisbauer96}).

\subsection{}\label{CorollaryZelmanowitz}
As a consequence we have that if $A$ is projective in $\sAB$ and $A^B$ large in $A$, then $A^B$ is regular and left self-injective if and only if  $A$ is injective and non-singular in $\sAB$, because the endomorphism ring of any self-injective polyform module is self-injective and regular by \cite[11.1]{wisbauer96}.

\section{Relative semisimple extensions}

Let $A\subseteq B$ be an extension of $k$-algebras. An element $c=\sum_i c_i \otimes c^i \in B\otimes_A B$ which is $B$-centralising, i.e.  $bc=cb$ for all $b\in B$ is called a {\it Casimir element} for $B$ over $A$ (see \cite{Schauenburg} for the terminology). We say that a Casimir element acts unitarily on an element $m$ of a left $B$-module $M$ if $\left(\sum_i c_ic^i\right) \cdot m = m$.

\begin{prop}\label{CasimirTraceOne}
Let $A\subseteq B$ be an extension with additional module structure and suppose that $B$ has a Casimir element over $A$ that acts unitarily on $A$, then  the following hold:
\begin{enumerate}
 \item[(1)] $c$ acts unitarily on any module in $\sigma [_BA]$.
 \item[(2)] The $k$-linear map
$ \HomX{A}{M}{N} \rightarrow \HomX{B}{M}{N}$ with $f  \mapsto \tilde{f}:[m \mapsto \sum c_i \cdot f(c^i \cdot m) ]$
splits the embedding $\HomX{B}{M}{N} \subseteq \HomX{A}{M}{N}$ for any $N,M\in \sigma[_BA]$.
\end{enumerate}
\end{prop}

\begin{proof} Let $\gamma := \sum c_i c^i$ and  $\alpha : B \longrightarrow A$ with  $(b)\alpha=b\cdot 1$. 
Then  $\alpha$ is left $B$-linear and $(a)\alpha=a$ for any $a\in A$.
For all $a\in A$ we have $ac = \sum ac_i \otimes c^i = \sum c_i \otimes c^ia = ca$.
Then also $\left(\sum ac_ic^i\right)\alpha = \left(\sum c_ic^ia\right)\alpha$ holds. Thus 
$$(*)\:\:\: a = a (\gamma)\alpha
= \left( \sum ac_ic^i\right)\alpha
= \left( \sum c_ic^ia \right)\alpha
= \left(\sum c_ic^i\right) \cdot (a)\alpha = \gamma \cdot a.$$
(1) Let  $M \in \sigma[_BA]$. Then there exists a set $\Lambda$ and a $B$-submodule
$I \subseteq A^{(\Lambda)}$, such that $M$ is isomorphic to a $B$-submodule of $A^{(\Lambda)}/I$.
We identify $M$ with a submodule of $A^{(\Lambda)}/I$.
Let  $m \in M$, then there are elements  $a_\lambda \in A$ for $\lambda \in \Lambda$ such that $m = ( a_\lambda )_\Lambda + I$. Now it follows with  $(*)$:
$$\gamma \cdot m = \gamma \cdot [ ( a_\lambda )_\Lambda + I ] =
(\gamma \cdot a_\lambda)_\Lambda + I = (a_\lambda)_\Lambda + I = m.$$
(2) Obviously $\tilde{f}$ is $B$-linear for all $f:M\longrightarrow N$  since $c$ is a Casimir element. If $f$ was already $B$-linear, then using $(1)$ we get for all $m\in M$:
 $$\tilde{f}(m) = \sum c_i \cdot f(c^i \cdot m) = \left(\sum c_i c^i\right) \cdot f(m) = f(m),$$
 i.e. $\tilde{f}=f$ showing that the embedding splits.
\end{proof}
\subsection{}
$M$ is a {\it $(B,A)$-semisimple } $B$-module if any short exact sequence in $\sigma[{_BM}]$ that splits as left $A$-module, also splits as left $B$-module (see \cite[page 170]{wisbauer96}). Recall that Hirata and Sugano called a ring extension $A\subseteq B$ a semisimple extension if $B$ is  $(B,A)$-semisimple (see \cite{HirataSugano}). 

\begin{cor} \label{CasimirTraceOne3}
If $B$ has a Casimir element $c$ which acts unitarily on $A$,  then  $A$ is a $(B,A)$-semisimple $B$-module and for any $M\in \sigma[_BA]$ 
\begin{itemize}
 \item If $M$ is $N$-projective as $A$-module for $N\in \sigma[_BA]$, then $M$ is also $N$-projective as $B$-module.
\item If $M$ is $N$-injective as $A$-module for $N \in B$-Mod, then $M$ is also $N$-injective as $B$-module.
\end{itemize}
In particular $A$ is projective in $\sigma[_BA]$.
\end{cor}
 
\begin{proof}
Let  $\pi: M\longrightarrow N$ be a projection in $\sAB$ with $\pi(n)=n$. Then for any 
$n\in N$:
$$ \tilde{\pi}(n) = \sum c_i \cdot \pi(c^i \cdot n) = \left( \sum c_ic^i\right) \cdot \pi(n) = \pi(n) = n.$$
Thus $\tilde{\pi}$ splits the embedding of $N$ into $M$ as $B$-module.
In the same way one proves the statements (1). For (2) note that if $f:U\rightarrow M$ is $B$-linear, where $U$ is a $B$-submodule of $N$, then there exists an $A$-linear map $g:N\rightarrow M$ such that $g_{\mid_U} = f$. Set as before $\tilde{g}:N\rightarrow M$ which is $B$-linear. Then $\tilde{g}(u)= \left( \sum c_i c^i \right) \cdot f(u) =f(u)$ .
\end{proof}

 
\subsection{}
In \cite{HirataSugano} Hirata and Sugano called a ring extension $A\subseteq B$ {\it separable} if there exists a Casimir element $c=\sum_i c_i \otimes c^i$ such that $\sum_i c_ic^i = 1$.

\begin{cor}\label{CasimirCorollar} Let $A\subseteq B$ be an extension with additional module structure such that there exists a Casimir element in $B$ which acts unitarily on $A$, then 
\begin{enumerate}
 \item If $A$ is a semisimple artinian ring, then $A$ is semisimple $B$-module.
 \item If  $A$ is von Neumann regular and $_AB$ is finitely generated, then  
\begin{itemize}
\item $A$ is a regular module in $\sigma[_BA]$;
\item $A^B$ is a regular ring and
\item $(-)^B$ defines a Morita equivalence between $A^B$-Mod and $\sigma[_BA]$ .
\end{itemize}
\item If  $\sigma[_BA] = B$-Mod, then  $A\subseteq B$ is a semisimple extension.
\end{enumerate}
\end{cor}

\begin{proof}
(1) Is clear since $A$ is $(B,A)$-semisimple.\\

(2) Since $_AB$ is finitely generated, $A$ is finitely presented in $\sigma[_BA]$.
 If  $B\cdot a$ is a cyclic $B$-submodule of  $A$, then by hypothesis $B\cdot a$  is also finitely generated as left $A$-module and hence a direct summand of $A$ as left $A$-module. Thus $B\cdot a$ is also a direct summand of $A$ as left $B$-module since $A$ is $(B,A)$-semisimple. By \ref{RegularExtensionsModuleStructure} $A$ is a regular module in $\sigma[_BA]$.
Also by \ref{RegularExtensionsModuleStructure}  we have that $\EndX{B}{A}\simeq A^B$  is regular and $A$ is a progenerator in $\sigma[_BA]$ with equivalence $\HomX{B}{A}{-}\simeq (-)^B: \sigma[_BA] \longrightarrow \EndX{B}{A}\simeq A^B.$

(3) If $A$ is a subgenerator in $B$-Mod, then $B\in\sigma[_BA]$ is itself $(B,A)$-semisimple.
\end{proof}

\section{Applications to Hopf algebra actions}
Let $H$ be a Hopf algebra over $k$ acting on an algebra $A$, i.e. $A$ is a left $H$-module algebra. The smash product of $A$ and $H$ is denoted by $\AH$ whose underlying $k$-module is $A\otimes_k H$ and whose multiplication is defined by
$$ (a\# h ) ( b\# g) = \sum_{(h)} a (h_1\cdot b) \# h_2g, $$
where $\Delta(h)=\sum_{(h)} h_1 \otimes h_2$ is the comultiplication of $h$.
Then $A\subseteq \AH=:B$ is an extension with additional module structure  whose module action is given by $a\#h \cdot b = a (h\cdot b)$. The subring of invariants is $A^B = A^H=\{ a \in A \mid h\cdot a = \epsilon(h)a \:\:\forall h\in H\}$. For more details on Hopf algebra action we refer to \cite{Montgomery}.

\subsection{Regularity of the subring of invariants}

From \ref{RegularRelativeFixRing} we get a characterisation of the regularity of the subring of invariants of $A$.

\begin{prop}\label{RegularFixRing} Let $A$ be a $k$-algebra with Hopf action $H$. Then $A^H$ is regular if and only if  $A$ is a semi-projective left $\AH$-module such that any cyclic left ideal generated by an $H$-invariant element is generated by an $H$-invariant idempotent.
\end{prop}

\subsection{}

In order to ensure that $A$ is a finitely presented object in $\sAH$ we will assume some finiteness conditions on $H$ or on its action. We say that a Hopf algebra $H$ {\it acts finitely} on a $k$-algebra $A$ if the image of the defining action $H\rightarrow \EndX{k}{A}$ is a finitely generated $k$-module or equivalently if $\AH/\mathrm{Ann}_{\AH}(A)$ is finitely generated as left $A$-module.
Recall that a $k$-algebra $A$ is called {\it affine} if it is finitely generated as $k$-algebra. 

Denote by $\epsilon:H\rightarrow k$ the counit of $H$. We need the following Lemma:

\begin{lem}\label{AffineHopfAlgebras} Let $H$ be a Hopf algebra over $k$ that is affine as $k$-algebra, then $\Ker{\epsilon}$ is a finitely generated left ideal.\end{lem}
\begin{proof}
Suppose that $H$ is affine and let $\cB \subseteq H$ be a finite set of elements which generate $H$ as a $k$-algebra.
We will show that $\Ker{\epsilon} = \sum_{b\in \cB} H(b-\epsilon(b))$. Obviously the right hand side is included in the left hand side. Note that for any word(=product) $\omega=b_1\cdots b_m$ with $b_i\in \cB$ we might set $a_0=\epsilon(\omega)$, $a_m=\omega$ and $a_i = b_1\cdots b_i \epsilon(b_{i+1}\cdots b_m)$ for $0< i < m$ and conclude that $\omega-\epsilon(\omega) \in \sum_{i=1}^m H(b_i-\epsilon(b_i))$, since as a telescopic sum we have
$$ \omega - \epsilon(\omega) = \sum_{i=1}^m a_i - a_{i-1} = \sum_{i=1} b_1\cdots b_{i-1}\epsilon(b_{i+1}\cdots b_n) (b_i - \epsilon(b_i)).$$
Take any element $h \in \Ker{\epsilon}$. Then there exist $\lambda_i\in k$ and words $\omega_i$ in $\cB$ such that 
$$h =h-\epsilon(h) =\sum_{i} \lambda_i [ \omega_i - \epsilon(\omega_i) ] \in \sum_{b\in \cB} H(b - \epsilon(b)).$$
Thus $\Ker{\epsilon}$ is finitely generated.
\end{proof}

In the telescopic sum argument in the proof of the last Lemma we made use of the fact that the counit $\epsilon$ of a Hopf algebra is an algebra homomorphism. We do not know whether this Lemma holds true for affine weak Hopf algebras.
\subsection{}
From Proposition \ref{RegularExtensionsModuleStructure} we deduce the next result:

\begin{thm}\label{RegularModuleAlgebras} If $H$ is an affine $k$-algebra or acts finitely on $A$, then $A$ is a finitely presented in  $\sAH$ and the following statements are equivalent:
\begin{enumerate}
 \item $A$ is {\it $H$-regular}, i.e. any finitely generated $H$-stable left ideal is generated by an $H$-invariant element;
\item $A$ is a projective generator in $\sAH$ and any cyclic left ideal generated by an $H$-invariant element is generated by an $H$-invariant idempotent.
\item $A^H$ is von Neumann regular and $(-)^H: \sAH \rightarrow A^H$-Mod is an equivalence.
 \item $A^H$ is von Neumann regular and $A$ is a projective generator in $\sAH$.
 \item $A$ is a regular module in $\sAH$.
\end{enumerate}
\end{thm}
\begin{proof}
 Once we showed that $A$ is finitely presented in $\sAH$, the result follows from \ref{RegularExtensionsModuleStructure}. 
If $H$ acts finitely on $A$, then we might substitute $\AH$ by $B= \AH/\mathrm{Ann}_{\AH}(A)$ which is finitely generated as left $B$-module. Also $\alpha$ lifts to a map $\ov{\alpha}:B\rightarrow A$ and splits as left $A$-module map. Thus $\Ker{\ov{\alpha}}$ is finitely generated as left $A$-module and thus as left ideal of $B$, i.e. $A$ is finitely presented in $B$-Mod and hence in $\sigma[{_BA}]=\sAH$.

On the other hand suppose that $H$ is affine, then $\Ker{\epsilon}$ is a finitely generated left ideal by Lemma \ref{AffineHopfAlgebras}. For the module algebra $A$ and $\Ker{\alpha:\AH\rightarrow A}$, we have that if $x=\sum_{i=1}^n a_i \# h_i \in \Ker{\alpha}$, then
$$x = \sum_{i=1}^n a_i \# h_i - \left(\sum_{i=1}^n a_i\epsilon(h_i) \right)\# 1 =  \sum_{i=1}^n a_i \# (h_i -\epsilon(h_i)) \in A\# \Ker{\epsilon}.$$ Thus $\Ker{\alpha} = A \# \Ker{\epsilon} = \sum_{b\in \cB} \AH (1\#(b - \epsilon(b)))$ is a finitely generated left ideal of $\AH$ and therefore $A$ is finitely presented.
\end{proof}

\subsection{}
Note that the notion of regularity used here is different from the concept of an $H$-regular module algebra as defined by \cite{s.zhang}. There the author define  an element $a$ of an $H$-module algebra $A$ to be $H$-regular if $a\in (H\cdot a)A(H\cdot a)$ and calls $A$ $H$-regular if every element is $H$-regular.

\subsection{The envelopping Hopf algebroid}
In general a Hopf action does not extend to the envelopping algebra $A^e$ unless $H$ is cocommutative.
In order to study the two-sided $H$-stable ideals of a Hopf module algebra $A$ with Hopf action $H$, one defines a new product on the tensor product $A^e\otimes H$ as follows:
$$[(a\otimes b)\bowtie h ][(a'\otimes b')\bowtie h'] = \sum_{(h)} a(h_1 \cdot a') \otimes (h_3\cdot b')b \bowtie h_2h' $$
for all $a\otimes b, a'\otimes b' \in A^e$ and $h,h'\in H$.
This construction had been used by the author in \cite{Lomp} (see also \cite{Lomp05}) in order to define the central closure of a module algebra $A$ as the self-injective hull of $A$ as $\AAH$-module and had also been used by Connes and Moscovici in  \cite{ConnesMoscovici}). A similar construction had been used by L.Kadison in \cite{Kadison} which in \cite{PanaiteOystaeyen} was shown to be isomorphic to the construction by Connes-Moscovici. Following Kadison, we denote this algebra on $A^e\otimes H$ by $\AAH$ and call it the {\it envelopping Hopf algebroid} of $A$ and $H$. For any left $A^e\bowtie H$-module $M$ denote by 
$$Z(M)^H := \{m \in M \mid am=ma \wedge hm=\epsilon(h)m \forall a \in A, h\in H\}.$$
Then since $A\subseteq \AAH$ is again an extension with additional module structure we have that  $Z( - )^H$ is a functor from $A^e\bowtie H \rightarrow Z(A)^H$-Mod and that  $$\HomX{A^e\bowtie H}{A}{M} \rightarrow Z(M)^H \:\:\: f \mapsto (1)f$$
is a functorial isomorphism. Note that $Z(A)^H:= Z(A)\cap A^H\simeq \EndX{\AAH}{A}$.

From \ref{RegularRelativeFixRing} we get a characterisation of the regularity of the subring of central invariants of $A$.
\begin{cor}\label{RegularFixRingBi}
$Z(A)^H$ is regular if and only if  $A$ is a semi-projective left $\AAH$-module such that any cyclic ideal generated by a central $H$-invariant element is generated by a central $H$-invariant idempotent.
\end{cor}

\subsection{}
As before we need to ensure that $A$ is a finitely presented object in $\sAAH$ in order to apply \ref{RegularExtensionsModuleStructure}.

\begin{lem}
If $A$ and $H$ are affine $k$-algebras, then $A$ is a finitely presented module in $\sAAH$.
\end{lem}
\begin{proof}
Consider $\alpha: \AAH \rightarrow A$ by $a\otimes b \bowtie h \mapsto a\epsilon(h)b$. For any $x=\sum_{i=1}^n a_i\otimes b_i \bowtie h_i \in \Ker{\alpha}$ we have $\sum_{i=1}^n a_i\epsilon(h_i)b_i = 0$. Hence
\begin{eqnarray*}
 x &=& \sum_{i=1}^n a_i\otimes b_i \bowtie h_i  - \left( \sum_{i=1}^n a_ib_i\epsilon(h_i) \right) \otimes 1 \bowtie 1  + \left[ \sum_{i=1}^n a_i\otimes b_i \bowtie \epsilon(h_i) - \sum_{i=1}^n a_i\otimes b_i \bowtie \epsilon(h_i)\right]\\
 &=& \sum_{i=1}^n a_i  \otimes b_i \bowtie [h_i - \epsilon(h_i)]  + \sum_{i=1}^n a_i\epsilon(h_i) [1\otimes b_i - b_i\otimes 1] \bowtie 1\\
&\in & A^e \bowtie \Ker{\epsilon} + A \Ker{\mu} \bowtie 1
\end{eqnarray*}
where $\Ker{\mu}$ is the augmentation ideal of the envelopping algebra, i.e. the kernel of the multiplication map $\mu:A^e\rightarrow A$.
Hence we see that $\Ker{\alpha}$ is generated as left ideal of $\AAH$ by elements of $1\otimes \Ker{\epsilon}$ and $\Ker{\mu}\bowtie 1$.
 It is well-known that $\Ker{\mu}$ is finitely generated left ideal of $A^e$ if $A$ is affine and by \ref{AffineHopfAlgebras} it follows that $\Ker{\epsilon}$ is finitely generated if $H$ is affine.
\end{proof}

\subsection{$H$-biregular module algebras}

The last statement,  \ref{RegularExtensionsModuleStructure} and \ref{CharacterisationBiregularity} yield the main result in this section which generalises   \cite[1.2]{Castella} from group actions to Hopf actions.

\begin{cor} Let $A$ and $H$ be affine $k$-algebras, then the following statements are equivalent:
\begin{enumerate}
\item[(a)] $A$ is $H$-biregular, i.e. every finitely generated $H$-stable two-sided ideal of $A$ is generated by a central $H$-invariant idempotent.
\item[(b)] $A$ is a projective generator in $\sAAH$ and any ideal generated by a central $H$-invariant element is generated by an idempotent central $H$-invariant element.
\item[(c)] $A$ is a regular module in $\sAAH$.
\item[(d)] $Z(A)^H$ is von Neumann regular and one of the following statements hold:
\begin{enumerate}
	\item[(i)] the functor $Z(-)^H: \sAAH \longrightarrow Z(A)^H$-Mod is an equivalence.
 	\item[(ii)] $A$ is a projective generator in $\sAAH$.
	\item[(iii)] every maximal $H$-stable ideal $M$ of $A$ can be written as  $M=[Z(A)^H\cap M] A$.
	\item[(iv)] $A_m$ is $H$-simple for any maximal ideal $m$ of $Z(A)^H$.
\end{enumerate}
\end{enumerate}
\end{cor}

\subsection{Relative semisimple extension}

Let $G$ be a finite group acting on an algebra $A$. The condition that $|G|$ is invertible is frequently used in the study of group actions because it implies that $A\subseteq A*G$ is a separable extension. The weaker condition on $A$ of having an {\it element of trace $1$}, i.e. an element $z\in A$ such that $t\cdot a = \sum_{g\in G} (g\cdot a) = 1$, where $t=\sum_{g\in G} g$, implies at least the projectivity of $A$ as $A*G$-module. Here we will analyse those concepts and carry them over to Hopf algebra actions. 

The antipode of a Hopf algebra $H$ is denoted by $S$. An element $t\in H$ is called a right(resp. left) integral in $H$ if $th=t\epsilon(h)$ 
(resp. $ht=\epsilon(h)t$) for all $h\in H$. it is well-known that $\sum_{(t)} S(t_1) \otimes t_2 h =  \sum_{(t)} hS(t_1) \otimes t_2$ for all $h\in H$.

\begin{prop}\label{PropCentralTrace}
Let $H$ be a Hopf algebra and $A$ a left $H$-module algebra.
Suppose $H$ has a non-zero right integral $t$ and $A$ admits a central element $z$ such that $S(t)\cdot z = 1$, then 
$$c:= \sum_{(t)} (1\# S(t_1)) \otimes (z\# t_2)$$
is a Casimir element of $\AH$ that acts unitarily on $A$. Hence $A$ is a semisimple $(\AH,A)$-module.
\begin{enumerate}
\item  $A\subseteq \AH$ is a semisimple extension if $A^H \subseteq A$ is a $H^\ast$-Galois extension, i.e. $A$ is a generator in $\AH$-Mod.
\item $A\subseteq \AH$ is separable if $z\in A^H$ or  $H$ is cocommutative
\end{enumerate}
\end{prop}

\begin{proof}
The element $\sum_{(t)} 1\# S(t_1) \otimes z\# t_2$ is a Casimir element in  $\left( \AH \right) {\otimes_A} \left( \AH \right) $ because for all $a\# h \in \AH$:

\begin{eqnarray*}
c(a\#h) &=& \sum_{(t)} (1 \# S(t_1))\otimes (z \# t_2)(a\# h) \\
&=& \sum_{(t)} (1 \# S(t_1))\otimes z(t_2\cdot a)\# t_3h\\
&=& \sum_{(t)} (S(t_2)\cdot (t_3\cdot a))\# S(t_1) \otimes z\# t_4h\\
&=& (a\#1) \left( \sum_{(t)} 1\# S(t_1) \otimes z\# t_2h \right)\\
&=& (a\#1) \left( \sum_{(t)} 1\# hS(t_1) \otimes z\# t_2 \right)\\
&=& (a\#h \left( \sum_{(t)} 1\# hS(t_1) \otimes z\# t_2 \right) = (a\# h) c
\end{eqnarray*}
Moreover
$$\sum_{(t)} (1 \# S(t_1)) (z \# t_2) \cdot 1 = \sum_{(t)} S(t_1) \cdot ( z (t_2 \cdot 1)) = S(t) \cdot z = 1.$$
Thus $c$ acts unitarily on $A$.  By  \ref{CasimirTraceOne}, $A$ is a semisimple $(\AH,A)$-module.

If $A/A^H$ is a $H^\ast$-Galois extension, then  $\sigma[{_\AH}A]=\AH$ and the claim follows from  \ref{CasimirCorollar}.

Note that  $\mu(c) = \sum_{(t)} S(t_2)z \# S(t_1)t_3.$
If $z\in Z(A)^H$, then $\mu(c) = \sum_{(t)} z\# S(t_1)\epsilon(t_2)t_3 =  \epsilon(S(t))z \# 1 = S(t)z \#1 = 1\#1.$
If  $H$ is cocommutative, then $\mu(c) = \sum_{(t)} S(t_1)z \# S(t_2)t_3 = S(t)x\# 1 = 1\# 1.$
Hence in both cases $\AH$ is separable over $A$.
\end{proof}

\subsection{}
If the antipode is bijective and $A$ has a central element of trace $1$, i.e. $z\in Z(A)$ with $t\cdot z=1$ for a left integral $t$ of $H$, then $t'=S^{-1}(t)$ is a right integral, and  $S(t')\cdot z = 1$ holds, i.e. the condition of \ref{PropCentralTrace} is fulfilled.
\subsection{}
The observation that $A\subseteq \AH$ is separable if $z \in  Z(A)^H$ or $H$ cocommutative, can also be found in \cite[Theorem 1.11]{CohenFishman92} or  \cite{CohenFishmanErrata}, but under the hypothesis of $H$ being a Frobenius $k$-algebra  and thus finitely generated and projective as $k$-module. Note that the existence of a left (or right) integral forces a Hopf algebra in many cases to be finitely generated although there are examples of non-finitely generated ones (see \cite{Lomp_Integrals}).

\subsection{Regularity of smash products}

In \cite{AlfaroAraRio} the authors studied the regularity of skew group rings. They showed in particular that a skew group ring $A*G$ is regular if $A$ is regular, $G$ is locally finite and for every finite subgroup $H$ of $G$ there exists a central element of $H$-trace $1$. In this section we will show how much of their arguments go over to smash products.

\subsection{}

First note the following Corollary that we get from \ref{PropCentralTrace}:

\begin{cor}\label{centralStraceOneRegular} Let $H$ be a Hopf algebra acting on a regular module algebra $A$.
Assume that there exists a right integral $t$ of $H$ and a central element $z$ such that $S(t)\cdot z = 1$.
\begin{enumerate}
\item[(1)] If $H$ acts finitly on $A$, then $A$ is regular  in $\sigma[{_\AH A}]$,
$A^H$ is regular and $A^H$-Mod is Morita-equivalent to $\sigma[{_\AH A}]$.
\item[(2)] If $z$ $H$-invariant or $H$ is cocommutative or  $A/A^H$ is $H^\ast$-Galois and $_kH$ is finitely generated, then $\AH$ is regular.
\end{enumerate}
\end{cor}

\begin{proof}
 (1) If we substitute $\AH$ by $B=\AH/\mathrm{Ann}(A)$, then $_AB$ is finitely generated and $c=\sum_{(t)} 1\# S(t_1) \otimes z\# t_2 \in \AH \otimes_A \AH$ can be lifted to $c'\in B\otimes_A B$ which still acts unitarily on $A$. By Corollary \ref{CasimirCorollar}, $A$ has the properties stated above.

(2) if $z\in A^H$ or $H$ is cocommutative, then by \ref{PropCentralTrace}, $A\subseteq \AH$ is separable and hence $\AH$ is regular as $A$ was regular. In case $A/A^H$ is $H^*$-Galois we have that $\sAH = \AH$-Mod and by \ref{PropCentralTrace}, $A \subseteq \AH$ is a semisimple extension. Since $_kH$ is finitely generated, $_A\AH$ is finitely generated. Hence any cyclic left ideal $I$ of $\AH$ is finitely generated as left $A$-submodule of $\AH$. Since $A$ is regular $I$ is a direct summand of $\AH$ and as $A\subseteq \AH$ is semisimple, it is also a direct summand of $\AH$ as left ideal.
\end{proof}


\subsection{Locally finite Hopf algebras}

Call  an extension $A\subseteq B$ {\it locally separable} if every element of $B$ is contained in an intermediate algebra $A\subseteq C \subseteq B$ such that $C$ is a separable extension of $A$ (see also A.Magid's definition \cite{Magid}).
Of course, if $A\subseteq B$ is locally separable and $A$ is regular, then $B$ is regular, because any element $x\in B$ is contained in a separable extension $C$ of $A$. And if $A$ is regular, then also $C$. Thus $x=xyx$ for some $y\in C\subseteq B$. Hence $B$ is regular.
The characterisation of regular group rings $k[G]$ can actually be stated as $k\subseteq k[G]$ being locally separable and $k$ being regular.
Alfaro, Ara and del Rio proved in  \cite[Theorem 1.3]{AlfaroAraRio} that if $G$ is a locally finite group acting on a regular ring $A$ such that for every finite subgroup $H$ there exists a central element of trace $1$ with respect to $H$, then the skew group ring $A*G$ is regular. We will slightly generalize there result to Hopf algebra actions by showing that the  hypotheses of their result imply that $A\subseteq A*G$ is locally separable.

A group is called locally finite if any finitely generated subgroup is finite.

\begin{defn}
Let $H$ be a Hopf algebra over $k$.
A Hopf algebra is called {\it locally finite} if any finite set $X\subseteq H$ is contained in a Hopf subalgebra of $H$  which contains a non-zero right integral.
\end{defn}

Note that any Hopf algebra $H$ which is free as a module over $k$ is finitely generated as $k$-module if and only if it contains a non-zero right integral (see \cite{Lomp_Integrals}). A group ring $k[G]$ is of course locally finite if $G$ is locally finite.
\subsection{}

We are now in position to generalize \cite[Theorem 1.3]{AlfaroAraRio}:

\begin{cor}
Let $H$ be a locally finite Hopf algebra over $k$ acting on a $k$-algebra $A$ such that for any Hopf subalgebra $K$ of $H$ that contains a non-zero right integral $t$, there exists a central element $z_t \in A$ with $S(t)\cdot z_t = 1$.
If $H$ is cocommutative or $z_t \in Z(A)^H$ for all right integrals $t$, then $A\subseteq \AH$ is locally separable. Hence if $A$ is regular, so is $\AH$.
\end{cor}

\begin{proof}
Let $x:=\sum_i^k a_i\# h_i \in \AH$. By hypothesis $K:=<\{h_1, \ldots, h_k\}>$ contains a non-zero right integral $t$. By \ref{PropCentralTrace}(2), $A\subseteq A\#K$ is separable, and hence regular if $A$ was regular.
\end{proof}

\subsection{}
As a consequence we have that if $H$ is a cocommutative Hopf algebra acting on a commutative regular $k$-algebra having a central element of trace $1$, then $\AH$ is regular (which partly generalizes  \cite[Corollary 2.5]{AlfaroAraRio}).  It had been shown in \cite[2.4]{AlfaroAraRio}, that if a skew-group ring $A*G$ is regular, then is also $A$. This is not anymore true for smash products as it is easily seen by the fact, that for any finite dimensional Hopf algebra $H$ over a field $k$, the smash product $H\# H^* \simeq M_n(k)$ is isomorphic to a semsimple  artinian ring, whether $H$ is semisimple or not. However we have that if $H$ is an $n$-dimensional cosemisimple Hopf algebra over a field $k$ acting on an algebra $A$ such that $\AH$ is regular, then $A$ is regular. Simply because by the Blattner-Montgomery duality one has $(\AH) \# H^* \simeq M_n(A)$ and since $H^*$ is separable over $k$, we have $\AH\#H^*$ being separable over $\AH$ inducing regularity on $M_n(A)$ and hence on $A$.

\section{Regularity and injectivity of the subring of invariants}

Note that from Zelmanowitz result  \ref{Zelmanowitz_polyform} we get

\begin{cor} Let $A$ be an $H$-module algebra that is projective in $\sigma[A]$.
If $A^H$ is large in $A$ then  $A^H$ is left self-injective and von Neumann regular if and only if $A$ is a self-injective left $\AH$-module which is non-singular in $\sAH$. In this case $A$ is also $H$-semiprime.
\end{cor}

\begin{proof}
 The equivalence of the statements follows verbatim from \ref{Zelmanowitz_polyform}. If $I$ is an $H$-stable ideal of $A$ with $I^2=0$, then $(I^H)^2=0$. But since $A^H$ is regular, $I^H=0$ and since $A^H$ is large $I=0$.
\end{proof}

\subsection{}
To compare the injectivity of $A$ and its subring of invariants, we need the following Lemma which is probably known:

\begin{lem}\label{InjectivitiyLemma} Let $S\subseteq T$ be rings such that $T_S$ is flat. If $T$ is
left self-injective, then so is $S$.
\end{lem}
\begin{proof}
Let $I$ be a left ideal of $S$, denote the inclusion map by $f:I
\longrightarrow S$ and
let $g:I\longrightarrow S$ be an $S$-linear map.
Let $\gamma: T \longrightarrow T\otimes_S S$ be the canonical isomorphism.
Then $\gamma$ is left $T$-linear and $\gamma|_{TI} : TI \longrightarrow
T\otimes_S I$ is also an isomorphism of
left $T$-modules. Let $\tilde{f}:=\gamma (1\otimes f)\gamma^{-1} : TI
\longrightarrow T$.
As $T_S$ is flat, $\tilde{f}$ is injective. Also set
$\tilde{g}:= \gamma (1\otimes g)\gamma^{-1}: TI \longrightarrow T$.
Then we can consider the following diagram with exact rows, where
$\imath: S \longrightarrow T$  denotes the inclusion map (which is of course
just $S$-linear):
$$\begin{CD}
0 @>>> I @>f>> S \\
@.@V{\imath}VV @V{\imath}VV \\
0 @>>> TI @>{\tilde{f}}>> T \\
@.@V{\tilde{g}}VV \\
@.T
\end{CD}$$
As $T$ is left self-injective, there there exists a $T$-linear map
$\tilde{h}: T \longrightarrow T$ such that $\tilde{f}\tilde{h} = \tilde{g}$.
Hence the outer trapezoid also commutes, i.e
$f\imath \tilde{h} = \imath\tilde{g}$. Since for all $x \in I: \: (x)\imath
\tilde{g}=(x)g$
we may identify $\imath\tilde{g}$ with $g$
and take $h:=\imath \tilde{h}$ to be the desired $S$-linear map.
\end{proof}

\subsection{}
We will finish with the following result on the transfer of regularity and injectivity to the subring of invariants of a module algebra which should be compared to \cite[Theorem A]{GOPV}.

\begin{cor}
 Let $H$ be Hopf algebra acting on $A$. Suppose $H$ has a right integral $t$ and $A$ has a central element $z$ such that $S(t)\cdot z = 1$.
If $A$ is regular and left self-injective ring, then $A^H$ is regular and left self-injective.
\end{cor}

\begin{proof}  Since $A$ is $(\AH.A)$-semisimple by \ref{PropCentralTrace}, $A$ is semi-projective as $\AH$-module. 
Take any $x\in A^H$, then $Ax$ is a direct summand in $A$ and by relative semisimplicity also a direct summand as $\AH$-submodule. Thus by \ref{InjectivitiyLemma} $A^H$ regular. Now it follows from  (1) that $A^H$ is also left self-injective. 
\end{proof}

\end{document}